\documentclass{amsart}
\usepackage{latexsym, bbm, amssymb, amsmath, enumerate}
\usepackage[abbr,dcucite]{harvard}
\usepackage[dvips]{graphicx}
\usepackage{setspace}
\usepackage{paralist}
\usepackage{ifpdf}
\usepackage[capposition=top]{floatrow}

\usepackage[show]{ed}

\usepackage{tikz}

\makeatletter
\def\subsection{\@startsection{subsection}{1}%
  \z@{.5\linespacing\@plus.7\linespacing}{.3\linespacing}%
  {\normalfont\scshape}}
\makeatother

\theoremstyle{plain}
\newtheorem{theorem}{Theorem}
\newtheorem{corollary}[theorem]{Corollary}
\newtheorem{lemma}[theorem]{Lemma}
\newtheorem{proposition}[theorem]{Proposition}
\theoremstyle{definition}

\allowdisplaybreaks[1]

\newcommand{\N}{\mathbb{N}}

\renewcommand{\P}{\mathbb{P}}

\renewcommand{\SS}{\mathcal{S}}
\newcommand{\E}{\mathbb{E}}
\newcommand{\F}{\mathcal{F}}

\newcommand{\G}{\mathcal{G}}
\newcommand{\T}{\mathcal{T}}

\newcommand{\hyp}{\mathbf{H}}

\newcommand{\1}{\mathbbm{1}}

\DeclareMathOperator{\Var}{Var}
\DeclareMathOperator{\Cov}{Cov}

\let\oldmarginpar\marginpar
\renewcommand\marginpar[1]{\-\oldmarginpar[\raggedleft\footnotesize #1]%
{\raggedright\footnotesize #1}}

\definecolor{mypurple}{rgb}{.3,0,.5}

\citationmode{full}

\begin{document}

\title[]%
{Optimal On-Line Selection
of an \\ Alternating Subsequence: \\
A Central limit theorem}
\author[]
{Alessandro Arlotto and J. Michael Steele}

\thanks{A. Arlotto:  The Fuqua School of Business, Duke University, Durham, NC, 27708.
Email address: \texttt{alessandro.arlotto@duke.edu}}

\thanks{J. M.
Steele: The Wharton School, Department of Statistics, Huntsman Hall
447, University of Pennsylvania, Philadelphia, PA 19104.
Email address: \texttt{steele@wharton.upenn.edu}}

\begin{abstract}

    We analyze the optimal policy for the sequential selection of an alternating subsequence from a sequence of $n$
    independent observations from a continuous distribution $F$, and we prove a central limit theorem for the number of selections
    made by that policy. The proof exploits the backward recursion of dynamic programming and assembles a detailed understanding of
    the associated value functions and selection rules.

    \medskip


    \medskip

\end{abstract}

\date{first version: December 6, 2012; this version: June 9, 2013}

\maketitle


\section{Introduction}

In the problem of on-line selection of an alternating subsequence,
a decision maker observes a sequence of independent random variables $\{X_1, X_2, \ldots, X_n\}$
with common continuous distribution $F$, and the task is to select a subsequence
such that
$$
X_{\theta_1} < X_{\theta_2} > X_{\theta_3} < \cdots \gtrless X_{\theta_k}
$$
where the indices $1 \leq \theta_1 < \theta_2 < \theta_3 < \cdots < \theta_k \leq n$ are stopping times with respect to the $\sigma$-fields $\F_i = \sigma\{X_1, X_2, \ldots, X_i \}$, $1 \leq i \leq
n$. In other words, at time $i$ when the random variable $X_i$ is first observed, the decision maker has to choose to accept $X_i$ as a member of the alternating sequence that is under construction,
or choose to reject $X_i$ from any further consideration.

We call such a sequence of stopping times a \emph{feasible policy}, and we denote the set of all such policies
by $\Pi$. For any $\pi \in \Pi$, we then let $A^o_n(\pi)$ denote the number of selections made by $\pi$ for the realization
$\{X_1,X_2, \ldots, X_n\}$, i.e.
$$
A^o_n(\pi) = \max\left\{ k : X_{\theta_1} < X_{\theta_2} >  \cdots \gtrless X_{\theta_k} \text{ and }
1 \leq \theta_1 < \theta_2 < \cdots < \theta_k \leq n\right\}.
$$
It was found in \citeasnoun{ArlCheSheSte:JAP2011} that for each $n$ there is a unique
policy $\pi^*_n \in \Pi$ such that
$$
\E[A^o_n(\pi^*_n)] = \sup_{\pi \in \Pi} \E[A^o_n(\pi)],
$$
and it was proved that optimal mean  $\E[A^o_n(\pi^*_n)]$ can be tightly estimated. Specifically, one has
\begin{equation}\label{eq:MeanAsymptotics}
   \E[A^o_n(\pi^*_n)]  =  (2 - \sqrt{2}) n + O(1).
\end{equation}
Here, our main goal is to show that $A^o_n(\pi^*_n)$ satisfies a central limit theorem.

\begin{theorem}[Central Limit Theorem for Optimal On-line Alternating Selection]\label{th:AlternatingCLT}
There is a constant $0<\sigma^2 < \infty$ such that
$$
\frac{A^o_n(\pi^*_n) - ( 2 - \sqrt{2}) n}{\sqrt{n}} \Longrightarrow N(0,\sigma^2) \quad \quad \text{as } n \rightarrow \infty.
$$
\end{theorem}

The exact value of  $\sigma^2$ is not known, but $\sigma^2$ has a representation as an infinite series
and Monte Carlo calculations\footnote{Numerical estimates are obtained discretizing the state
space with a grid size of $10^{-4}$ and performing $5 \cdot 10^5$ repetitions.
The standard error for the estimate of $\sigma^2$ is $6.19\times 10^{-4}$.} suggest that $\sigma^2 \sim 0.3096$. The
determination of a closed-form expression for $\sigma^2$ remains an open problem. It may even be a tractable problem, though it is unlikely to be easy.

\subsection*{Motivation: History and Connections}

The theory of alternating sequences has ancient roots. It began with
the investigations of Euler on alternating permutations, and, through a long evolution, it
has become an important part of combinatorial theory \citeaffixed{Sta:CM2010}{cf.}. The
probability theory of alternating sequences is much more recent, and its main problems fit into two basic
categories: problems of  global selection and problems of sequential  selection.

In a problem of global selection (or an \emph{off-line} problem), one sees the whole sequence $\{X_1, X_2, \ldots, X_n\}$,
and the typical challenge is to understand the distribution of
length of the longest alternating subsequence under various probability models.
For example, when $\{X_1, X_2, \ldots, X_n\}$ is a random permutation of the integers $[1:n]$,
explicit bivariate generating functions were used by \citeasnoun{Wid:EJC2006}, Pemantle
\citeaffixed[p.~568]{Sta:PROC2007}{cf.}, and \citeasnoun{Sta:MMJ2008} to obtain central limit theorems.
Simpler probabilistic derivations of these results were then developed by  \citeasnoun{HouRes:EJC2010}
and \citeasnoun{Rom:DMTCS2011}. These authors exploited the close connection between the length of the longest alternating subsequence
and the number of local extrema of a sequence, a link that is also relevant to
local minima problems studied in computer science \citeaffixed{BanEpp:AACO2012}{e.g.}
and to similar structures in the theory of turning point tests (e.g., \citename{BroDav:SPRI2006}, \citeyear*{BroDav:SPRI2006}, p.~312, or
\citename{Hua:MIT2010}, \citeyear*{Hua:MIT2010}, Section 1.2).

The theory of on-line alternating subsequences is of more recent origin,
but it is closely tied to some classic themes of applied probability.
In the typical on-line decision problem, a decision maker considers $n$ random values in sequential order
and must decide whether to accept or reject each presented value at the time of its first presentation.
In the most famous such a problem, the decision maker gets to make only a single choice,
and his goal is to maximize the probability
that the selected value is the best out of all $n$ values.
\citeasnoun{Cay:ET1875} considered a problem of this kind, but the modern development of the theory began in earnest with
notable studies by \citeasnoun{Lin:APPLSTAT1961} and \citeasnoun{Dyn:SSSR1963}.
\citeasnoun{Sam:DEKKER1991} gives a thoughtful survey of the work on related problems through the 1980's, and connections to
more recent work are given by \citeasnoun{KriSamCah:AAP2009}, \citeasnoun{BucJaiSin:LNCS2010},
and \citeasnoun{BatHaiZad:SPRI2010}.

In more complex problems, the decision maker typically makes multiple sequential selections from the sequence of
presented values, and the objective is to maximize the expected number of selected elements, subject to
a combinatorial constraint.
For example, one can consider the optimal sequential
selection of a \emph{monotone} subsequence. This on-line selection
problem was introduced in \citeasnoun{SamSte:AP1981}, and
it has been analyzed more recently
in \citename{Gne:JAP1999} \citeyear{Gne:JAP1999,Gne:CPC2000,Gne:IMS2000},
\citeasnoun{BarGne:AAP2000}, \citeasnoun{BruDel:SPA2001} and \citeasnoun{ArlSte:CPC2011}.

The present investigation is particularly motivated by \citeasnoun{BruDel:SPA2004}, where a central limit theorem
is proved for the sequential selection of a monotone subsequence when the number $N$ of values offered to the
decision maker is a Poisson random variable that is independent of the sequence of the offered values. The methodology
of \citeasnoun{BruDel:SPA2004} is tightly bound with the theory of Markov processes and Dynkin's formula, while the present method
leans heavily on the Bellman equation and explicit estimates of the decision functions.

\subsection*{Organization of the Analysis}

The proof of Theorem \ref{th:AlternatingCLT} rests on a sustained investigation of the value functions
that are determined by the Bellman equation for the alternating selection problem.
The optimal policy $\pi^*_n$ is determined in turn by the
time-dependent threshold functions $\{g_n, g_{n-1}, \ldots, g_1\}$ that tell us when to
accept or reject a newly presented value.
Inferences from the Bellman equation almost inevitably require inductive arguments, and the
numerical calculations summarized in Figure \ref{fig:thresholds} are a great help in
framing appropriate induction hypotheses.

In Section \ref{se:optimal-policy}, we frame the selection problem as a dynamic program,
and we summarize a few results from earlier work. The main observation is that, by symmetry, one
can transform the natural
Bellman equation into an equivalent recursion that is much simpler. We also note that
the value functions determined by the reduced recursion have a useful
technical feature, which we call the property of \emph{diminishing returns}.

Sections \ref{se:geometry-value-function} through \ref{se:optimal-policy-at-infty} develop the geometry of the value and
threshold functions. Even though the alternating subsequence problem
is rather special, there are generic elements to its analysis, and our intention is to make these elements as visible as possible.
Roughly speaking, one frames concrete hypotheses based on the suggestions of Figure \ref{fig:thresholds} (or its analog),
and one proves these hypotheses by inductions that are driven by the Bellman equation.
While the specific inferences are unique to the problem of alternating selections, there is still some robustness to the
pattern of the proof.

Sections \ref{se:A-infty} and \ref{se:CLT-optimal-A} exploit the geometrical characterization of the threshold functions
to obtain information about the distribution of $A^o_n(\pi^*_n)$, the number of selections made by the optimal policy for
the problem with time horizon $n$.
The main step here is the introduction of a horizon-independent policy $\pi_\infty$ that
is determined by the limit of the threshold functions that define $\pi^*_n$. It is
relatively easy to check that the number of selections $A^o_n(\pi_\infty)$
made by this policy is a Markov additive functional of a
stationary, uniformly ergodic, Markov chain. Given this observation,
one can use off-the-shelf results to confirm that the
central limit theorem holds for $A^o_n(\pi_\infty)$, provided that one shows that the variance of
$A^o_n(\pi_\infty)$ is not $o(n)$.
We then complete the proof of Theorem \ref{th:AlternatingCLT} by showing that there is a coupling under which
$A^o_n(\pi^*_n)$ and $A^o_n(\pi_\infty)$
are close in $L^2$; specifically, we show that one has
$\parallel A^o_n(\pi^*_n)-A^o_n(\pi_\infty) - \E\left[ A^o_n(\pi^*_n)-A^o_n(\pi_\infty) \right] \parallel_2=o(\sqrt{n})$.

\section{Dynamic Programming Formulation}\label{se:optimal-policy}

We first note that since the distribution $F$ is continuous and since the problem is unchanged
if we replace $X_i$ by $U_i=F^{-1}(X_i)$, we can assume without
loss of generality that the $X_i$'s are uniformly distributed on $[0,1]$. The main task now is to exploit
the symmetries of the problem to obtain a tractable version of the Bellman equation.

We proceed recursively, and, for $1\leq i \leq n$, we let $S_i$ denote the value of the last member of the subsequence
selected up to and including time $i$. We also set $R_i=0$ if $S_i$ is a local minimum of
$\{S_0,S_1,\ldots,S_i\}$, and we set $R_i=1$ if $S_i$ is a local
maximum. Finally, to initialize our process, we set $S_0 = 1$ and $R_0 = 1$,
and we note that the process $\{(S_i, R_i): 0\leq i \leq n\}$ is Markov.

At time $i$, the decision to accept or reject the new observation $X_i$ depends only on two quantities: (1) the state of the selection
process before the presentation of  $X_i$; this is represented by
the pair $(S_{i-1},R_{i-1})$
and (2) the number of observations $k$ that were yet to be seen
before the presentation of $X_i$, i.e. $k = n-i+1$.

One can now characterize the optimal policy $\pi^*_n$ through
these state variables and an associated dynamic programming equation (or Bellman
equation) for the value function.
We let $v_{k}(s,r)$ denote the expected number of optimal alternating selections when
the number of observations yet to be seen is $k$, and the state of the selection process is given
by the pair $(s,r)$.
If $k=0$, then we set $v_0(s,r) \equiv 0$ for all $(s,r) \in [0,1]\times\{0,1\}$.
Otherwise, for  $ 1 \leq i \leq n$, states $S_{i-1}=s$ and $R_{i-1} = r$, and residual time $k = n-i+1$, we have
the Bellman equation
\begin{equation*}
v_{k}(s,r)=
\left\{
 \begin{array}{ll}
 \!\!sv_{k-1}(s,0) +\int_{s}^1 \max\left\{v_{k-1}(s,0),1+v_{k-1}(x,1)\right\}\,dx
 & \hbox{if $r=0$} \\
 \!\!(1-s)v_{k-1}(s,1) +\int_0^{s} \max\left\{v_{k-1}(s,1),1+v_{k-1}(x,0)\right\}\,dx
 & \hbox{if $r=1$}.\\
 \end{array}
 \right.
\end{equation*}

To see why this equation holds, first consider the case when $r=0$ (so the next selection needs to be a local maximum).
With probability $s$, we are presented with a value, $X_{i}$, that is less than the
previously selected value. In this case, we do not have the opportunity to make a selection, we reduce the
number of observations yet to be seen to $k-1$,
and this contributes the term $s v_{k-1}(s,0)$ to our equation.

Next, consider the case when $r=0$ but $s < X_{i} \leq 1$. In this case, one must decide to select $X_{i}=x$, or to reject it.
If we do not select the value $X_{i}=x$, then the expected number of subsequent selections equals $ v_{k-1}(s,0)$.
If we do select  $X_{i}=x$, then we
account for the selection of $x$ plus the expected number of optimal subsequent selections, which together equal $1 + v_{k-1}(x,1)$.
Since $X_{i}$ is uniformly distributed in $[s,1]$ the expected optimal contribution
is given by the second term of our Bellman equation (top line).
The proof of the second line of the Bellman equation is completely
analogous.

One benefit of indexing the value functions $v_{k}(\cdot,\cdot)$ by
the ``time-to-go'' parameter $k$
is that, by the optimality principle of dynamic
programming, the selection problem for a sequence of size $n$ embeds automatically into the selection problem for a sequence of size
$n+1$. As a consequence, we can consider the infinite sequence of value functions $\{v_{k}(\cdot,\cdot),1 \leq k < \infty \}$.
It is also useful to observe that these
value functions
satisfy an intuitive\footnote{A formal proof of \eqref{eq:symmetry} is given in \citeasnoun[Lemma 3]{ArlCheSheSte:JAP2011}.}
symmetry property:
\begin{equation}\label{eq:symmetry}
v_k(s,0) = v_k(1-s, 1) \quad \quad \text{for all } 1 \leq k < \infty \text{ and all } s \in [0,1],
\end{equation}
so we can define the \emph{single-variable value function} $v_k(y)$, $1 \leq k < \infty$, by setting
$$
v_k(y) \equiv v_k(y,0) = v_k(1-y, 1) \quad \quad \text{for all } 1 \leq k < \infty \text{ and all } y \in [0,1].
$$
Now, when we replace the bivariate value function $v_{k}(\cdot,\cdot)$ in the earlier Bellman equation
with the corresponding value of the univariate value function $v_k(\cdot)$,
we obtain a much nicer recursion:
\begin{equation}\label{eq:Bellman-FINITE-flipped}
v_{k}(y)  = y \, v_{k-1}(y) +\int_{y}^1 \max\left\{ v_{k-1}(y),1+ v_{k-1}(1 - x)\right\}\,dx.
\end{equation}
Here, we have  that $v_0(y)\equiv 0$ for all $y \in [0,1]$,
and we note that the map $y \mapsto v_k(y)$ is continuous and differentiable on $[0,1]$,
and it satisfies the boundary condition $v_k(1) = 0$ for all $ 1 \leq k < \infty$.
In this reduced setting,
the state of the selection process is simply given by
the value $y$, rather than the pair $(s,r)$.

The key benefit of the reduced
Bellman equation \eqref{eq:Bellman-FINITE-flipped} is that it leads to a simple rule for
the optimal acceptance or rejection of a newly presented observation.
Specifically, if we set
\begin{equation}\label{eq:g-FINITE-star}
g_{k}(y) = \inf \{x \in [y, 1]:~ v_{k-1}(y) \leq 1+ v_{k-1}(1 - x) \},
\end{equation}
then a value $x$ is an optimal selection if and only if $g_k(y) \leq x$.
For $1 \leq k < \infty$, we then call the function $g_k : [0,1] \rightarrow [0,1]$
the \emph{optimal threshold function} if it satisfies the variational characterization
given by \eqref{eq:g-FINITE-star}.
We will see shortly that the value function $y \mapsto v_{k}(y)$ is
strictly decreasing on $[0,1]$, a fact that will imply that
$g_k(y)$ is uniquely determined for each $y \in [0,1]$.

The optimal threshold functions $\{g_n, g_{n-1}, \ldots, g_1\}$
give us a useful representation  for the number
$A^o_n(\pi^*_n)$ of selections made by the optimal policy $\pi^*_n$.
Specifically, if we set $Y_0 \equiv 0$ and define the sequence $Y_1, Y_2, \ldots$,
by the recursion
\begin{equation*}
Y_i =
\begin{cases}
    Y_{i-1} & \text{if $X_i< g_{n-i+1}(Y_{i-1})$}\\
    1 - X_{i} & \text{if $X_i\geq g_{n-i+1}(Y_{i-1})$},
\end{cases}
\end{equation*}
then we have that
\begin{equation}\label{eq:A-optimal}
A^o_{n}(\pi^*_n) = \sum_{i=1}^{n } \1\left( X_i \geq g_{n-i+1}(Y_{i-1}) \right),
\end{equation}
and, moreover, by the principle of optimality of dynamic programming, we have
$$
\E[A^o_n(\pi^*_n)] = v_n(0) \quad \quad \text{for each } n \geq 1.
$$

The representation \eqref{eq:A-optimal} also tells us that
$A^o_{n}(\pi^*_n)$ is a sum of functions of a time inhomogeneous Markov chain.
The analysis of this inhomogeneous additive functional calls for a reasonably  detailed
understanding of both the threshold functions $\{g_k(\cdot): 1 \leq k < \infty\}$, and the value functions $\{v_k(\cdot): 1 \leq k < \infty\}$.

A technical fact that will be needed shortly is that, for each $1 \leq k < \infty$, the value function $v_k(\cdot)$
satisfies the bound
\begin{equation}\label{eq:diminishing-return}
v_{k-1}(u) - v_{k-1}(1-y) \leq v_{k}(u) - v_{k}(1-y)  \text{ for all } y\in [0,1/2] \, \text{and } u\in [y, 1-y].
\end{equation}
This bound reflects a restricted \emph{principle of diminishing returns};
a proof of \eqref{eq:diminishing-return} is given by \citeasnoun[Lemma 4]{ArlCheSheSte:JAP2011}.

Given the dynamic programming formulation provided here,
the results in this paper can be read independently of our earlier work.
Still, for the purpose of comparison,
we should note that the notation used here simplifies our earlier one in some significant ways.
For example, we now take $k$  to be number of observations \emph{yet to be seen}, and this gives us
the pleasing formulation \eqref{eq:Bellman-FINITE-flipped} of the Bellman equation.
We also write $g_k(y)$ for the optimal threshold function when there are $k$ observations yet to be seen, and this replaces the earlier,
more cumbersome, notation $f^*_{n-k+1,n}(y)$.

\section{Geometry of the Value and Threshold Functions}\label{se:geometry-value-function}

Figure \ref{fig:thresholds} gives a highly suggestive picture of the individual
threshold functions $g_k(\cdot),$ and it foretells much of the story about how they behave as $k \rightarrow \infty$.
Analytical confirmation of these suggestions is the central challenge.
The path to understanding the threshold functions goes through the value functions, and
we begin by proving the very plausible fact that
the value functions are strictly decreasing.

\begin{figure}[t]
    \centering
    \caption{Threshold functions $g_{k}(y)$, $1\leq k \leq 10$, and their limit as $k \rightarrow \infty$ for $y \in [0, 0.35]$.}

    \floatfoot{Solid lines plot the threshold functions $g_{k}$, $1\leq k \leq 10$,
    for values of $y$ in the range $[0, 0.35]$.
    We have $g_1(y) = g_2(y)=y$ for all $y \in [0,1]$, and
    the piecewise smooth graphs of $g_3$, $g_4$, and $g_5$ are explicitly labeled
    with their index placed just below the curve. For $k=6,7,...,10$ the $g_k$  are indicated without labels.
    Each $g_{k}$ the meets the diagonal line at some point, and one has
    $g_{k}(y)=y$ on the rest of the interval $[0,1]$.  The plot suggests most of the analytical properties of the
    sequence $\{g_k: 1 \leq k <\infty\}$ that are needed for the proof of the central limit theorem. In particular,
    for each fixed $y \in [0,1]$  the sequence $k \mapsto g_k(y)$ is monotone non-decreasing.
    The dashed line represents the limit of $g_{k}(y)$ as $k \rightarrow \infty$; this limit is piecewise linear.
    } \label{fig:thresholds}
    \vspace{8pt}
    {
    \ifpdf
        {
        \includegraphics[angle=0,width=0.80\linewidth]{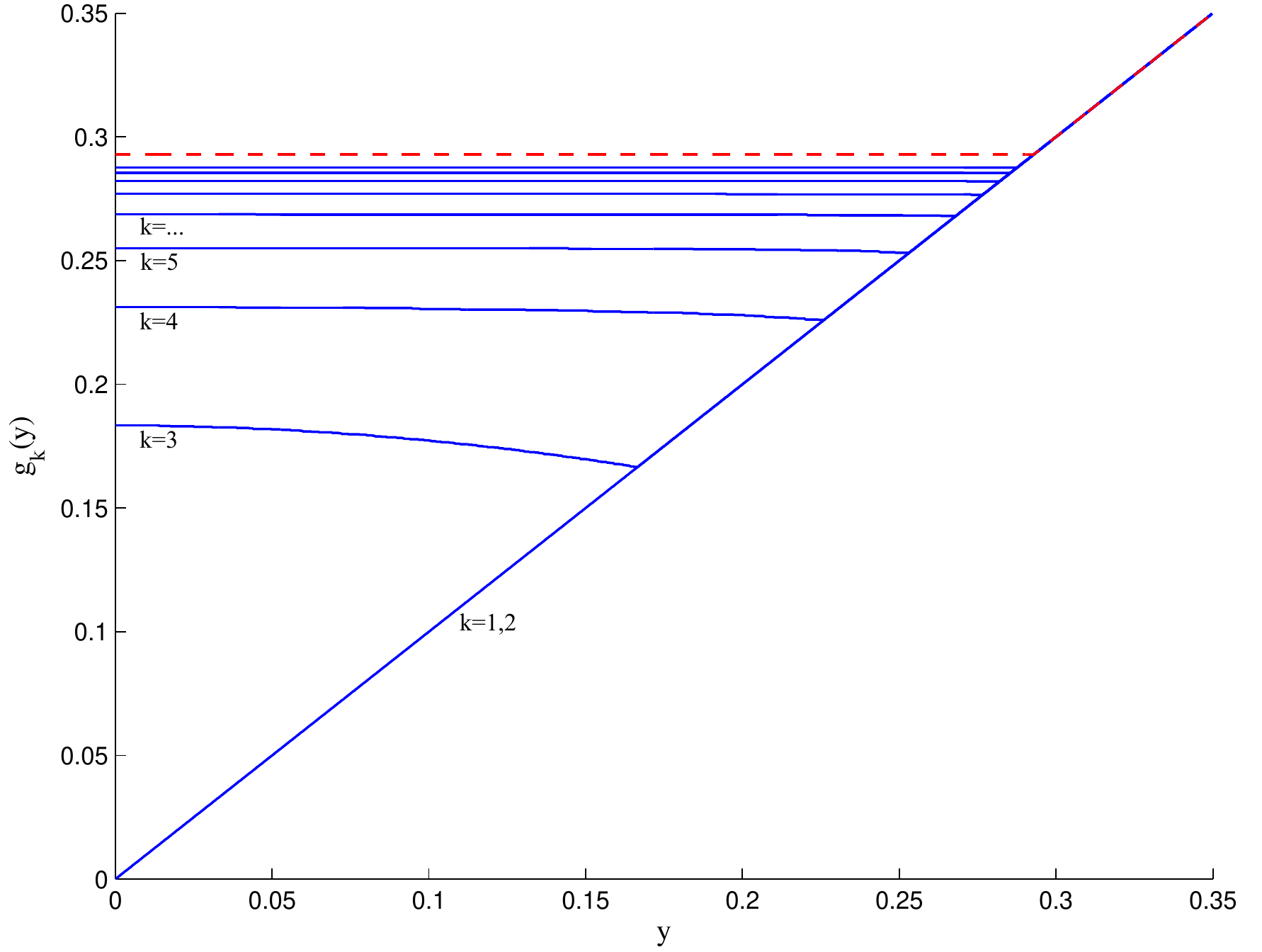}
        }
    \else
        {
        \includegraphics[angle=0,width=0.80\linewidth]{thresholds5.eps}
        }
    \fi
    }
\end{figure}

\begin{lemma}[Strict Monotonicity of the Value Functions]\label{lm:v-strictly-decreasing}
For each $1 \leq k < \infty$, the value function $y \mapsto v_k(y)$ defined by the
Bellman recursion \eqref{eq:Bellman-FINITE-flipped} is strictly decreasing on $[0,1]$.
\end{lemma}

This assertion is certainly intuitive and one may not feel any need for a proof.
Nevertheless, there is something
to be gained from a formal proof; specifically, one sees in a simple context how the Bellman equation can be used to propagate
a sequence of induction hypotheses.

\begin{proof}[Proof of Lemma \ref{lm:v-strictly-decreasing}]
We consider the sequence of hypothesis:
\begin{equation*}
\hyp_k: \quad v_k( y + \epsilon) < v_k (y) \quad \text{for all $y \in [0, 1)$ and all $\epsilon >0$ such that $y + \epsilon\leq 1$.}
\end{equation*}
Since $v_1(y) = 1-y$, $\hyp_1$ is true. For $k\geq 2$, we note by the Bellman recursion \eqref{eq:Bellman-FINITE-flipped}
that we have
\begin{align*}
v_k( y + \epsilon) - v_k (y)
& = (y+\epsilon) v_{k-1}( y + \epsilon) + \int_{y+\epsilon}^1 \!\!\!\! \max\{  v_{k-1}( y + \epsilon), 1 + v_{k-1}( 1-x ) \} \, dx \\
& \quad - y v_{k-1}( y ) - \int_{y}^1 \max\{ v_{k-1}( y ), 1 + v_{k-1}( 1-x ) \} \, dx\\
& \leq (y+\epsilon) v_{k-1}( y + \epsilon)  +  \int_{y+\epsilon}^1 \!\!\!\! \max\{  v_{k-1}( y ), 1 + v_{k-1}( 1-x ) \} \, dx \\
& \quad - (y + \epsilon) v_{k-1}( y ) - \int_{y +\epsilon}^1 \max\{ v_{k-1}( y ), 1 + v_{k-1}( 1-x ) \} \, dx\\
& = (y+\epsilon)\left\{ v_{k-1}( y + \epsilon)  - v_{k-1}( y ) \right\} < 0,
\end{align*}
where the first inequality of the chain follows from
$$ \epsilon \, v_{k-1}( y ) \leq \int_y^{y+\epsilon} \max\{ v_{k-1}( y ), 1 + v_{k-1}( 1-x ) \} \, dx$$
and the second inequality follows from $\hyp_{k-1}$. This completes the proof of $\hyp_k$ and of the lemma.
\end{proof}

Figure \ref{fig:thresholds} further suggests that
the threshold functions have a long interval of fixed points; the next lemma partially confirms this.

\begin{lemma}[Range of Fixed Points]\label{lm:v-initial-bound-improved}
For all $k \geq 1$ and $y \in [0,1]$, we have
\begin{equation}\label{eq:v-initial-bound-improved}
v_k(y) - v_k( 2/3 ) \leq v_k(0) - v_k( 2/3 ) \leq 1.
\end{equation}
In particular, for all $k \geq 1$, we have
\begin{equation}\label{eq:1/3}
g_k(y) = y \quad \text{ for all } y \in [1/3, 1]
\end{equation}
and
\begin{equation}\label{eq:1/3-upper-bound}
g_k(y) \leq 1/3 \quad \text{ for all } y \in [0, 1/3].
\end{equation}
\end{lemma}

\begin{proof}
The first inequality of \eqref{eq:v-initial-bound-improved} is trivial since the map $y \mapsto v_k(y)$ is strictly decreasing in $y$.
Also, the identities \eqref{eq:1/3} and \eqref{eq:1/3-upper-bound} are immediate from the  variational
characterization \eqref{eq:g-FINITE-star} and the bound \eqref{eq:v-initial-bound-improved}.

The real task is to prove the second inequality of \eqref{eq:v-initial-bound-improved}.
This time we use induction on the hypotheses given by
\begin{equation}\label{eq:induction-difference-0-2/3}
\hyp_k: \quad \quad v_k(0) - v_k( 2/3 ) \leq 1, \quad\quad \text{for } 1\leq k < \infty.
\end{equation}
As before $v_1(y)=1-y$, so $\hyp_1$ is trivially true. Now, when we apply the
Bellman recursion \eqref{eq:Bellman-FINITE-flipped} with $y = 0$ and $y = 2/3$, we get
\begin{eqnarray*}
    v_k (0) - v_k (2/3) & = & \int_0^1 \max\left\{ v_{k-1}(0), 1 + v_{k-1}(1-u)\right\} \, du \\
                        & -  & (2/3) v_{k-1}(2/3) - \int_{2/3}^1  \max\left\{ v_{k-1}(2/3), 1 + v_{k-1}(1-u)\right\} \, du,
\end{eqnarray*}
from which a change of variable gives
\begin{equation} \label{eq:difference-0-2/3}
v_k (0) - v_k (2/3) = \int_0^{1/3} I_1(u) \, du  + \int_{1/3}^1 I_2(u) \, du
\end{equation}
where $I_1(u)$ and $I_2(u)$ are defined by
$$
I_1(u) \equiv \max\left\{ v_{k-1}(0), 1 + v_{k-1}(u)\right\} -  \max\left\{ v_{k-1}(2/3), 1 + v_{k-1}(u)\right\}
$$
and
$$
I_2(u) \equiv \max\left\{ v_{k-1}(0) - v_{k-1}(2/3), 1 + v_{k-1}(u) - v_{k-1}(2/3) \right\}.
$$

For the first integrand, $I_1(u)$, we note that
\begin{align}\label{eq:first-integrandI1}
I_1(u) =  & \max\left\{ v_{k-1}(0) - v_{k-1}(2/3), 1 + v_{k-1}(u) - v_{k-1}(2/3)\right\} \\
          & -  \max\left\{ 0, 1 + v_{k-1}(u) - v_{k-1}(2/3)\right\}. \nonumber
\end{align}
The induction assumption $\hyp_{k-1}$ then tells us that
$$
 v_{k-1}(0)  - v_{k-1}(2/3) \leq 1,
$$
and the strict monotonicity of the value function $v_{k-1}(\cdot)$ on $[0,1]$ yields
\begin{equation*}
1 \leq  1 + v_{k-1}(u) - v_{k-1}(2/3) \quad \text{ for all } u \in [0, 1/3].
\end{equation*}
Thus, both the first and the second addend in \eqref{eq:first-integrandI1} equal the right maximand
and
\begin{equation}\label{eq:I1zero}
I_1(u)  = 0 \quad \text{ for all } u \in [0, 1/3],
\end{equation}
so the first integral in \eqref{eq:difference-0-2/3} vanishes.

To estimate $I_2(u)$, we note that $\hyp_{k-1}$ and
the monotonicity of $y \mapsto v_{k-1}(y)$ tell us that
\begin{enumerate}[(i)]
    \item if $u \in [1/3, 2/3]$, then
        $$
        I_2(u) = 1 + v_{k-1}(u) - v_{k-1}(2/3) \leq  1 + v_{k-1}(0) - v_{k-1}(2/3) \leq 2 \, \, \text{and}
        $$
    \item if $u \in [2/3, 1]$, then
        $$
        I_2(u) = \max\left\{ v_{k-1}(0) - v_{k-1}(2/3), 1 + v_{k-1}(u) - v_{k-1}(2/3) \right\} \leq 1.
        $$
\end{enumerate}
Now we just calculate
$$
 v_k (0) - v_k (2/3) = \int_{1/3}^{1} I_2(u) \, du \leq \int_{1/3}^{2/3} 2 \, du + \int_{2/3}^{1} 1 \, du = 1,
$$
and thus we  complete the proof of \eqref{eq:v-initial-bound-improved}.
\end{proof}

From Lemma \ref{lm:v-initial-bound-improved}, we know that a threshold function $g_k$ has many fixed points; in particular,
$g_k(y)=y$ if $y \in [1/3,1]$. Figure \ref{fig:thresholds} further suggests that much of the
geometry of $g_k$ is governed by its \emph{minimal} fixed point:
\begin{equation}\label{eq:xi_k}
\xi_k \equiv \inf\{y: g_k(y) = y\}.
\end{equation}

The value $\xi_k$ also has a useful policy interpretation. If the value $y$ of the last observation selected  is bigger than
$\xi_k$, then the decision maker follows a  greedy policy; he  accepts \emph{any} feasible arriving observation.
On the other hand, if $y < \xi_k$, the decision maker acts conservatively; his choices are governed by
the value of the threshold $g_k(y)$. Finally, if $y = \xi_k$, the greedy policy and the optimal policy agree.
This interpretation of $\xi_k$ is formalized in the next lemma, where we also prove that the
sequence $\{\xi_k: \, k = 1,2,\ldots\}$ is non-decreasing.

\begin{lemma}[Characterization of the Minimal Fixed Point]\label{lm:xi_k}
For $k \geq 3$, the minimal fixed point $\xi_k \equiv \inf \{y: g_k(y) = y\} $ is the unique solution to the equation
$$
v_{k-1}(y) - v_{k-1}(1-y) = 1.
$$
Moreover, the minimal fixed points form a non-decreasing sequence, so we have
\begin{equation}\label{eg:xi-monotone}
\xi_k \leq \xi_{k+1} \quad \text{ for all } \quad  k \geq 1.
\end{equation}
\end{lemma}

\begin{proof}
From the variational characterization of $g_k(\cdot)$, we have
$$
g_{k}(y) = \inf \{x \in [y, 1]:~ v_{k-1}(y) \leq 1+ v_{k-1}(1 - x) \},
$$
so if we set $\delta_k ( y ) \equiv v_{k-1}(y) - v_{k-1}(1-y)$,  then we have
\begin{equation}\label{eq:when-g-equals-y}
g_k(y) = y \quad \text{ if and only if} \quad \delta_k ( y ) \leq 1.
\end{equation}
The Bellman equation \eqref{eq:Bellman-FINITE-flipped} for $v_k(\cdot)$ and Lemma \ref{lm:v-strictly-decreasing}
tell us that the map $y \mapsto v_{k-1}(y)$ is continuous and strictly  decreasing
with $v_1(y) = 1-y$ and $v_2(y) = (3/2)(1-y^2)$.
Then, the function $\delta_k$ is continuous and strictly decreasing, and
for $k \geq 3$ we have
$ \delta_k(0) = v_{k-1}(0) \geq v_2(0) = 3/2 > 1$,
and
$
\delta_k(1) = - v_{k-1}(0)  < 0,
$
so, there is a unique value $y^*$ such that
$$
\delta_k ( y^* ) \equiv v_{k-1}(y^*) - v_{k-1}(1-y^*) = 1.
$$
Since the map $y \mapsto \delta_k ( y )$ is strictly decreasing, we can also write $y^*$ as
$$
y^* = \inf\{ y : v_{k-1}(y) - v_{k-1}(1-y) \leq 1 \} = \inf\{y: g_k(y) = y\} = \xi_k,
$$
where the second equality follows from \eqref{eq:when-g-equals-y} and the third equality comes from the definition
of $\xi_k$.

To prove the monotonicity property $\xi_k \leq \xi_{k+1}$ for all $k \geq 1$,
we first note that since $v_0(y) \equiv 0$ and $v_1(y) \equiv 1-y$, we have that $\xi_1 = \xi_2 = 0$.
Also, by Lemma \ref{lm:v-initial-bound-improved} we have for $k \geq 3$ that there is always a value
$0\leq y \leq 1/3$ such that  $g_k(y) = y$ so
\begin{align}
\xi_k  & = \inf\{y \in [0, 1/3]: g_k(y) = y\} \nonumber\\
       & = \inf\{ y \in [0, 1/3]: \delta_k(y) \equiv v_{k-1}(y) - v_{k-1}(1-y) \leq 1 \} \nonumber\\
       & \leq \inf\{ y \in [0, 1/3]: \delta_{k+1}(y) \equiv v_{k}(y) - v_{k}(1-y) \leq 1 \} \label{eq:xi-use-restricted-sm}\\
       & = \inf\{y \in [0, 1/3]: g_{k+1}(y) = y \}
       = \xi_{k+1},\nonumber
\end{align}
where the one inequality \eqref{eq:xi-use-restricted-sm} follows from the diminishing return property \eqref{eq:diminishing-return}.
\end{proof}

\section{A Second Property of Diminishing Returns}\label{se:DiminishingReturns2}

The value functions have a second property of diminishing returns that provides some  crucial help. Specifically,
we need it to show that the threshold functions $g_k(\cdot)$ increase with $1\leq k
<\infty$. This monotonicity moves us a long way toward an exhaustive
understanding of the asymptotic behavior of the threshold functions.

\begin{proposition}[Second Property of Diminishing Returns]\label{pr:DiminishingReturns2}
For all $k \geq 3$, the value functions defined by the Bellman recursion \eqref{eq:Bellman-FINITE-flipped}
satisfy the bound
\begin{equation}\label{eq:DiminishingReturns2}
v_{k-1}(y) - v_{k-1}(1-x) \leq v_{k}(y) - v_{k}(1-x) \text{ for all $ y \leq \xi_k$ and $x \in [y,  g_k(y)]$.}
\end{equation}
\end{proposition}
\begin{proof}
We again use induction to exploit the Bellman equation,
and this time the sequence of hypotheses is given by
$$
\hyp_{k}:  v_{k-1}(y) - v_{k-1}(1-x) \leq v_{k}(y) - v_{k}(1-x),  \text{ for all $ y \leq \xi_k$ and $x \in [y, g_k(y)]$.}
$$
We first prove $\hyp_3$, which we then use as the base case for our induction.
We recall that $v_1(y) = 1-y$ and, if we use the Bellman recursion \eqref{eq:Bellman-FINITE-flipped},
we obtain that $v_2(y) = (3/2)(1-y^2)$.
In turn, this implies $g_3(y) = \max\{ 1 - \sqrt{ 2/3 + y^2}, \, y\}$ and $\xi_3 = 1/6$.
To calculate $v_3 (y)$ we apply the Bellman recursion one more time, and we obtain a messier but still tractable
formula:
$$
v_3 (y) =   \begin{cases}
                 (3/2)(1 - y^2) + 3^{-3/2} (2 + 3y^2)^{3/2}         & \quad \text{if } y \leq 1/6 \\
                 (1/2)( 1 - y ) ( 4+ 5y + 2y^2)                              & \quad \text{if } y \geq  1/6.
            \end{cases}
$$
Thus, for $y \leq \xi_3 = 1/6$, we need to show
$$
v_{2}(y) - v_{2}(1-x) \leq v_{3}(y) - v_{3}(1-x)  \quad \text{for all $x \in [y, g_3(y)]$,}
$$
where $g_3 (y) = 1 - \sqrt{ 2/3 + y^2}$.
From our explicit formulas for $v_2(\cdot)$ and $v_3(\cdot)$, we have
$$
v_3(1-x) - v_2(1-x) = (5/2) x - 3 x^2 + x^3,
$$
and
$$
v_3(y) - v_2(y) =  3^{-3/2}(2 + 3y^2)^{3/2} \geq  \left( 2/3 \right)^{3/2} \approx 0.5443.
$$
Calculus shows that $(5 /2) x - 3 x^2 + x^3$ increases on $0 \leq x \leq  1 - \sqrt{2/3}$
and attains an endpoint maximum of  $(1/18) \left(9-\sqrt{6}\right) \approx 0.3640$.
Thus, we find
$$
v_3(1-x) - v_2(1-x) \leq (1/18)  (9-\sqrt{6}) < \left( 2/3 \right)^{3/2} \leq v_3(y) - v_2(y)
$$
for all $ y \leq 1/6$ and $ y \leq x \leq 1 - \sqrt{2/3 + y^2}$, completing the proof of $\hyp_{3}$.

We now suppose that $\hyp_k$ holds, and we seek to show $\hyp_{k+1}$.
First, from the variational characterization of $g_k(\cdot)$ and the definition of $\xi_k$, recall that
$$
1 \leq v_{k-1}(y) - v_{k-1}(1-x) \quad \text{ for } y \leq \xi_k \text{ and } x \in [y, g_k(y)],
$$
which, together with the induction assumption $\hyp_k$, implies
\begin{equation}\label{eq:hyp-k-explicit}
1 \leq v_{k-1}(y) - v_{k-1}(1-x) \leq v_{k}(y) - v_{k}(1-x) \quad  \text{ for } y \leq \xi_k \text{ and } x \in [y, g_k(y)].
\end{equation}
The second inequality in \eqref{eq:hyp-k-explicit} and
the variational characterization \eqref{eq:g-FINITE-star} give us
$$
g_k(y) \leq g_{k+1}(y) \quad \quad \text{ for all  } y  \leq \xi_k.
$$
Moreover, if $x \in [g_k(y), g_{k+1}(y)]$ the variational characterization of $g_{k+1}(\cdot)$ also gives
$$
    v_{k-1}(y) - v_{k-1}(1-x) \leq  1 \leq v_{k}(y) - v_{k}(1-x) \quad \text{ for $y \leq \xi_k$ and $x \in [g_k(y), g_{k+1}(y)]$,}
$$
which combines with \eqref{eq:hyp-k-explicit} to give the crucial inequality
\begin{equation}\label{eq:induction-sm2-extended}
    v_{k-1}(y) - v_{k-1}(1-x) \leq v_{k}(y) - v_{k}(1-x) \quad \text{ for $y \leq \xi_k$ and $x \in [y, g_{k+1}(y)]$}.
\end{equation}

From an application of the Bellman recursion \eqref{eq:Bellman-FINITE-flipped}
for $y \leq \xi_k$ and $x \in [y, g_{k+1}(y)]$, we obtain
\begin{align}\label{eq:integral-monotonicity-proof}
v_{k}(y) & - v_{k}(1-x) =  y \left( v_{k-1}(y) - v_{k-1}(1-x) \right) \nonumber\\
                        & + \int_y^{1-x} \max\left\{ v_{k-1}(y) - v_{k-1}(1-x) , 1 + v_{k-1}(1-u) - v_{k-1}(1-x) \right\} \, du.
\end{align}
If we now change variable in the last integral by replacing $u$ with $1-u$,
then the range of integration changes to $[x, 1-y]$ and
we can rewrite \eqref{eq:integral-monotonicity-proof} as
\begin{align*}
v_{k}(y) & - v_{k}(1-x) =  y \left( v_{k-1}(y) - v_{k-1}(1-x) \right) \\
                        & + \int_x^{1-x} \max\left\{ v_{k-1}(y) - v_{k-1}(1-x) , 1 + v_{k-1}(u) - v_{k-1}(1-x) \right\} \, du \nonumber\\
                        & + \int_{1-x}^{1-y} \max\left\{ v_{k-1}(y) - v_{k-1}(1-x) , 1 + v_{k-1}(u) - v_{k-1}(1-x) \right\} \, du. \nonumber
\end{align*}
In this last equation, we see that we can use our crucial inequality \eqref{eq:induction-sm2-extended}
to bound the first addend and the left maximand of the other two addends.
Moreover, since $ x \leq g_{k+1}(y) \leq 1/3$, we can appeal to the diminishing return property \eqref{eq:diminishing-return}
to bound the right maximand of the second addend. In doing so, we obtain
\begin{align} \label{eq:key-DiminishingReturns2}
v_{k}(y) & - v_{k}(1-x) \leq  y \left( v_{k}(y) - v_{k}(1-x) \right) \\
                        & + \int_x^{1-x} \max\left\{ v_{k}(y) - v_{k}(1-x) , 1 + v_{k}(u) - v_{k}(1-x) \right\} \, du \nonumber\\
                        & + \int_{1-x}^{1-y} \max\left\{ v_{k}(y) - v_{k}(1-x) , 1 + v_{k-1}(u) - v_{k-1}(1-x) \right\} \, du. \nonumber
\end{align}
We now observe that the monotonicity property of the map $u \mapsto v_{k-1}(u)$ for $u \in [1-x, 1-y]$
and the variational characterization of $g_{k+1}(\cdot)$ combine
to give
$$
1 + v_{k-1}(u) - v_{k-1}(1-x)  \leq 1 \leq v_{k}(y) - v_{k}(1-x)
$$
for all $y \leq \xi_k$  and  $x \in [y, g_{k+1}(y)]$.
Hence, the third integrand in \eqref{eq:key-DiminishingReturns2} satisfies the equality
$$
\max\left\{ v_{k}(y) - v_{k}(1-x) , 1 + v_{k-1}(u) - v_{k-1}(1-x) \right\} = v_{k}(y) - v_{k}(1-x),
$$
and an analogous monotonicity argument for $u \in [1-x, 1-y]$ also yields
$$
\max\left\{ v_{k}(y) - v_{k}(1-x) , 1 + v_{k}(u) - v_{k}(1-x) \right\} = v_{k}(y) - v_{k}(1-x).
$$
When we use the last two observations in \eqref{eq:key-DiminishingReturns2}
we obtain that
\begin{equation*}\label{eq:y-smaller-xik}
v_{k}(y) - v_{k}(1-x) \leq v_{k+1}(y) - v_{k+1}(1-x),  \text{ for all } y \leq \xi_k \text{ and } x \in [y, g_{k+1}(y)].
\end{equation*}

We now conclude our argument by considering values $y \in [\xi_k, \xi_{k+1}]$.
From the variational characterization of $g_{k+1}(\cdot)$ and the definition of $\xi_k$, we obtain
$$
v_{k-1}(y) - v_{k-1}(1-x) \leq 1 \leq v_{k}(y) - v_{k}(1-x) \quad \text{for $y \in [\xi_k, \xi_{k+1}]$ and $x \in [y, g_{k+1}(y)]$}
$$
which can be used instead of \eqref{eq:induction-sm2-extended} to construct
an argument similar to the earlier one and conclude that
$$
v_{k}(y) - v_{k}(1-x) \leq v_{k+1}(y) - v_{k+1}(1-x),  \text{ for } y \in[ \xi_k, \xi_{k+1}] \text{ and } x \in [y, g_{k+1}(y)],
$$
just as needed to complete the proof of \eqref{eq:DiminishingReturns2}.
\end{proof}

The usefulness of the property of diminishing returns in Proposition \ref{pr:DiminishingReturns2}
shows itself simply --- but clearly --- in the following corollary.

\begin{corollary}[Monotonicity of Optimal Thresholds]\label{cor:MonotoneThresholds}
For all  $y \in [0,1]$,  the threshold functions satisfy
\begin{equation}\label{eq:monotonicity-gk-inLemma}
g_k(y) \leq g_{k+1}(y) \quad \quad \text{ for all } k \geq 1, \text{and }
\end{equation}
\begin{equation}\label{eq:lower-bound-gk-inLemma}
1/6 \leq g_k(y) \quad \quad  \text{for all } k \geq 3.
\end{equation}
\end{corollary}
\begin{proof}
For $k=1,2$, we have $v_0(y) =0$ and $v_1(y) = 1 - y$, so that
$$
g_1(y) = g_2(y) = y.
$$
For $k = 3$, we have already noticed in the course of proving Proposition \ref{pr:DiminishingReturns2} that we have
$g_3(y) = \max\{ 1 - \sqrt{ 2/3 + y^2}\, ,\, y\}$,
so, in particular, $g_3(y) \geq 1/6$ for $y \in [0,1]$.
Finally, for $k > 3$, the bound \eqref{eq:DiminishingReturns2}
and the variational characterization \eqref{eq:g-FINITE-star} of
the threshold function give us \eqref{eq:monotonicity-gk-inLemma},
and this confirms the lower bound \eqref{eq:lower-bound-gk-inLemma}.
\end{proof}

We now pursue two further suggestions from Figure \ref{fig:thresholds}.
Specifically, we show that the limit function $g_\infty$ has exactly the piecewise linear shape that the figure suggests, and we also show
that the convergence to $g_\infty$ is uniform. The proof of these facts requires some additional regularity properties
that are discussed in the next section.

\section{Regularity of the Value and Threshold Functions}\label{se:first-derivatives}

The minimal fixed points give us a powerful guide to the geometry of the value function and its derivatives.
The connection begins with the Bellman recursion \eqref{eq:Bellman-FINITE-flipped} and the variational characterization
\eqref{eq:g-FINITE-star} which together give the identity
$$
v_k(y) = g_k(y) v_{k-1}(y) + \int_{g_k(y)}^1 \{1 + v_{k-1}(1-x) \} \, dx.
$$
If we now differentiate both sides with respect to $y$, we obtain the recursion for the
first derivative:
\begin{equation*}
v'_k(y) = g'_k(y) v_{k-1}(y) + g_k(y) v'_{k-1}(y) - g'_k(y) \left\{ 1 + v_{k-1}(1- g_k(y)) \right\}.
\end{equation*}
The definition of the minimal fixed point \eqref{eq:xi_k} and
the variational characterization \eqref{eq:g-FINITE-star} then give us
\begin{equation}\label{eq:v-of-g}
v_{k-1}(y)  =  1 + v_{k-1}(1- g_k(y)) \quad \quad \text{ if } y \leq \xi_k,
\end{equation}
so our recursion for $v'_k(\cdot)$ can be written more informatively as
\begin{equation}\label{eq:v-first-derivative}
v'_k(y) =   \begin{cases}
                g_k(y) v'_{k-1}(y)                                  & \quad \text{if $y \leq  \xi_k$} \\
                v_{k-1}(y)  - 1 - v_{k-1}(1- y) + y v'_{k-1}(y)     & \quad \text{if $y \geq  \xi_k$.}
            \end{cases}
\end{equation}
These relations underscore the importance of the minimal fixed points to the geometry of the value function,
and they also lead to useful regularity properties.

\begin{lemma}[Monotonicity Properties of the Derivatives]
For all $k\geq 1$, we have
\begin{align}
-1 \leq v'_k(y) &\leq v'_{k+1}(y) \leq 0\quad \quad \text{ for } y \in [0,1/3]  \label{eq:first-derivatives-up-1/3}
\text{ and }\\
v'_{k+1}(y) &\leq v'_{k}(y) \leq -1 \quad \quad \text{ for }\,\, y \in [1/2,1]  \label{eq:first-derivatives-down-1/2}.
\end{align}
\end{lemma}

\begin{proof}
We already know from Lemma \ref{lm:v-strictly-decreasing} that $y \mapsto v_k(y)$ is strictly decreasing,
so $v'_k(y)$ is non-positive on $[0,1]$.
Since $0 \leq g_{k}(y) \leq 1$, the top line of \eqref{eq:v-first-derivative} tells us that
\begin{equation}\label{eq:first-derivative-bound-pre-xi}
v'_{k-1}(y) \leq g_{k}(y) v'_{k-1}(y) = v'_{k}(y)  \quad \text{for } y \leq \xi_{k}.
\end{equation}
To cover the rest of the range in  \eqref{eq:first-derivatives-up-1/3}, we use induction on the sequence of hypotheses
$$
\hyp_k: \quad\quad v'_{k-1}(y) \leq v'_{k}(y), \quad\quad \text{for all } y \in [\xi_{k}, 1/3] \text{ and } 2 \leq k < \infty.
$$
For the base case $\hyp_2$, we have $\xi_2 = 0$, $v_1(y) = 1-y$, and  $v_2(y) = (3/2) (1-y^2)$. So
$$
v'_1(y) = -1 \leq - 3 y = v'_2(y) \quad \text{ if and only if } \quad y \leq 1/3,
$$
just as needed. Now taking  $\hyp_k$ as our induction assumption, we seek to prove $\hyp_{k+1}$.

First, for $y \in [\xi_{k}, 1/3]$, the second line of \eqref{eq:v-first-derivative} gives us $v'_{k}(\cdot)$.
By the diminishing return property \eqref{eq:diminishing-return},
the monotonicity $\xi_k \leq \xi_{k+1}$, and the induction assumption $\hyp_{k}$, we see for $y \in [\xi_{k+1}, 1/3]$
that
\begin{align*}
v'_{k}(y) & = v_{k-1}(y)  - 1 - v_{k-1}(1- y) + y v'_{k-1}(y) \\
          & \leq v_{k}(y)  - 1 - v_{k}(1- y) + y v'_{k}(y)
          = v'_{k+1}(y),
\end{align*}
completing the proof $\hyp_{k+1}$.
To complete the proof of \eqref{eq:first-derivatives-up-1/3}, one just needs to note that
the lower bound $-1 \leq v'_k(y)$ now follows from
$v'_1(y) = -1$ together with \eqref{eq:first-derivative-bound-pre-xi} and $\hyp_k$.

To prove \eqref{eq:first-derivatives-down-1/2}, we again use induction, but this time the sequence of hypothesis is given by
$$
\hyp_k: \quad \quad v'_{k}(y) \leq v'_{k-1}(y) \quad \quad \text{for } y \in [1/2, 1], \text{ and } 2 \leq k < \infty.
$$
As before, $v_1(y) = 1-y$ and $v_2(y) = (3/2) (1-y^2)$ so $v'_1(y) = -1$ and  $v'_2(y) = -3y$.
For $y \geq 1/2$, we then have
$$
v'_2(y) \leq -3/2 \leq -1 = v'_1(y),
$$
proving $\hyp_2$. As tradition demands, we again take $\hyp_k$ as our induction assumption, and we seek to prove $\hyp_{k+1}$.

Since $y \in [1/2, 1]$, we have $1-y \leq 1/2 \leq y$, so the diminishing return
property \eqref{eq:diminishing-return} gives us
\begin{equation}\label{eq:tempRSM}
v_{k-1}(1- y) - v_{k-1}(y) \leq v_{k}(1- y) - v_{k}(y).
\end{equation}

Next, recall the identity of the bottom line of \eqref{eq:v-first-derivative}, but, as you do so, replace $k$ by $k+1$. We can
then directly apply
\eqref{eq:tempRSM} and $\hyp_{k}$ to get
\begin{align*}
v'_{k+1}(y) & = v_{k}(y)  - 1 - v_{k}(1- y) + y v'_{k}(y) \\
            & \leq v_{k-1}(y)  - 1 - v_{k-1}(1- y) + y v'_{k-1}(y)
            = v'_{k}(y).
\end{align*}
This inequality completes the proof of $\hyp_{k+1}$ and confirms the lower bound of \eqref{eq:first-derivatives-down-1/2}.
For the upper bound of \eqref{eq:first-derivatives-down-1/2}, $v'_k(y) \leq -1 $ on $[1/2, 1]$, we just need to note that
it follows from the fact $v'_1(y) = -1$ and the validity of $\hyp_k$ for all $k\geq 1$.
\end{proof}

The smoothness of the value functions converts easily into a very useful Lipschitz equi-continuity property
of the threshold functions.

\begin{lemma}[Lipschitz Equi-Continuity of Threshold Functions]\label{lm:lip-g}
For all $k\geq 1$, we have
\begin{equation}\label{eq:g-Lip-1}
|g_k(y) -g_k(z)| \leq |y-z| \quad \text{for all } y,z \in [0,1].
\end{equation}
\end{lemma}
\begin{proof}
We first consider $y \in [0, \xi_k]$. In this case, we have that
identity \eqref{eq:v-of-g} holds, so, by its differentiation, we obtain
\begin{equation}\label{eq:first-derivative-g}
g'_k(y) = - \frac{ |v'_{k-1}(y)| }{ | v'_{k-1}( 1 - g_{k}(y)) | } \leq 0  \quad \quad \text{ for all } y \in [0, \xi_k].
\end{equation}
Moreover, since $y \in [0, \xi_k]$ we know that $y \leq 1/3$ so by \eqref{eq:1/3-upper-bound}
we have $g_k(y) \leq 1/3$, and hence by \eqref{eq:first-derivatives-down-1/2} we obtain
$1\leq  | v'_{k-1}( 1 - g_{k}(y)) |$. Consequently, \eqref{eq:first-derivative-g} gives us
\begin{equation}\label{eq:first-derivative-g-bound}
|g'_k(y)| \leq |v'_{k-1}(y)|  \quad \text{ for all } y \in [0, \xi_k],
\end{equation}
and \eqref{eq:first-derivatives-up-1/3} implies $|v'_{k}(y)|\leq 1$. Thus, at last, we have the uniform bound
\begin{equation}\label{eq:g-prime-bound}
\left| g'_k(y) \right| \leq 1 \quad \quad \text{ for all } y \in [0, \xi_k],
\end{equation}
which confirms the inequality \eqref{eq:g-Lip-1} for $y,z \in [0, \xi_k]$.
Also, for $y,z \in [\xi_k, 1]$ we have that \eqref{eq:g-Lip-1} trivially holds, so
if we choose $y < \xi_k < z$, the triangle inequality gives us
$$
| g_k(y) - g_k(z)| \leq  | g_k(y) - g_k(\xi_k) | + |g_k(\xi_k) - g_k(z)| \leq |y - z|,
$$
confirming that \eqref{eq:g-Lip-1} holds in general.
\end{proof}

\section{The Optimal Policy at Infinity}\label{se:optimal-policy-at-infty}

The minimal fixed points $\xi_k$, $1 \leq k < \infty$,  are non-decreasing and bounded
by $1/3$, so they have a limit
\begin{equation}\label{eq:def-xi}
\lim_{k\rightarrow \infty} \xi_k \stackrel{\text{def}}{=} \xi \leq 1/3.
\end{equation}
The threshold values $g_k(y)$, $1 \leq k < \infty$ are also non-decreasing
and bounded, so they have a pointwise limit $g_\infty(y)$. The next lemma characterizes
$g_\infty$ and gives a crucial bound on the uniform rate of convergence to $g_\infty$

\begin{proposition}[Characterization of Limiting Threshold]
For the limit threshold $g_\infty$, we have the formula
\begin{equation*}
g_\infty (y) =\max \{\xi, y\} \quad \text{for all } y \in [0,1].
\end{equation*}
Moreover, we have an exact measure of the uniform rate of convergence
\begin{equation}\label{eq:unif-conv-max1}
\max_{0 \leq y \leq 1} | g_k(y) - g_\infty(y) | = \xi - \xi_k \quad\quad \text{ for all } k \geq 1.
\end{equation}
\end{proposition}
\begin{proof}
We first fix $m$ and $y \in [0, \xi_m]$.
We then recall that $y \leq \xi_m \leq 1/3$ implies that $g_j(y) \leq 1/3$ for all $j \geq 1$.
Now, given $k \geq m$,  we can repeatedly apply the top line of \eqref{eq:v-first-derivative} to obtain
\begin{equation}\label{eq:v-first-derivative-iterated}
| v'_k (y) | = |v'_{m-1}(y)| \left(\prod_{j=m}^k g_j(y)\right) \leq 3^{m - k} |v'_{m-1}(y)|  \quad \text{for } y \in [0, \xi_m],
\end{equation}
and by \eqref{eq:first-derivatives-up-1/3} we have $|v'_{m-1} (y) | \leq 1$ for all $y \in [0, 1/3]$,
so \eqref{eq:first-derivative-g-bound} gives us more simply
\begin{equation}\label{eq:unf-bnd-g-prime-limit}
\max_{0 \leq y \leq \xi_m} | g'_k(y)|  \leq  3^{m - k} \quad \text{for all } k \geq m.
\end{equation}
Now, for any $y,z$ in $[0, \xi_m]$ we have $|g_k(y)-g_k(z)| \leq 3^{m-k}|y-z|$ so, letting $k\rightarrow \infty$,
we obtain that $g_\infty$ is constant on $[0, \xi_m]$ for each $m \geq 1$. Since $\xi_m \uparrow \xi$, there is a
constant $c$ such that $g_\infty(y)=c$ for all $y \in [0,\xi)$.

As Figure \ref{fig:thresholds} suggests, $c=\xi$ and this is easy to confirm.
Again we fix $m$, take $ k \geq m$, and note that by the triangle inequality and the
Lipschitz bound \eqref{eq:g-Lip-1} on $g_k$ we have
\begin{align*}
|g_\infty(\xi_m) -\xi_k| &\leq |g_\infty(\xi_m) -g_k(\xi_m)|+ |g_k(\xi_m) -g_k(\xi_k)|\\
&\leq |g_\infty(\xi_m) -g_k(\xi_m)| +  |\xi_m -\xi_k|.
\end{align*}
When  $k \rightarrow \infty$, $g_k(\xi_m)$ converges to $g_\infty(\xi_m)$
and $\xi_k$ to $\xi$ so we have
$$
|g_\infty(\xi_m) -\xi| \leq |\xi_m -\xi|.
$$
Since $g_\infty(\xi_m)=c$ does not depend on $m$ and since $|\xi_m -\xi| \rightarrow 0$ as $m \rightarrow \infty$,
we see that $g_\infty(\xi_m) = \xi$ for all $m\geq 1$ and consequently  $g_\infty(y) = \xi$ for all $ y \in [0,\xi]$.
Finally, for all $m\geq 1$, we also have $g_m(y)=y$ for each $y\in [\xi, 1]$, so the proof of the formula for $g_\infty$ is complete.

To prove \eqref{eq:unif-conv-max1}, we first note
$$
g_\infty(y) - g_k(y)=
\begin{cases}
\xi - g_k(y) & y \in [0,\xi_k], \\
\xi- y & y \in [\xi_k, \xi], \\
0 & y \in [\xi, 1].
\end{cases}
$$
By \eqref{eq:first-derivative-g}, $g_k(y)$ is strictly decreasing on $[0,\xi_k]$, so the
gap $g_\infty(y) - g_k(y)$ is maximized when $y=\xi_k$. This gap decreases linearly over
the interval $[\xi_k, \xi]$ and equals $0$ at $\xi$; consequently the maximal gap is exactly
equal to $\xi - \xi_k$.
\end{proof}

\section{The Central Limit Theorem for $A^o_{n}(\pi_\infty)$ Is Easy} \label{se:A-infty}

We now recall that $\xi$ denotes the limit \eqref{eq:def-xi} of the minimal fixed points, and
we define a selection policy $\pi_\infty$ for all $X_1, X_2, \ldots$ by taking the
(time independent) threshold function to be
$$
g_\infty(y) = \max\{\xi, y\} \equiv \xi \vee y.
$$
If $A^o_n(\pi_\infty)$ counts the number of selections made by policy $\pi_\infty$ up to and including time $n$, then we have
the explicit formula
\begin{equation}\label{eq:A-infinity}
A^o_{n}(\pi_\infty) = \sum_{i=1}^{n } \1\left( X_i \geq  \xi \vee Y'_{i-1}  \right),
\end{equation}
where one sets $Y'_0 = 0$, and one defines $Y'_i$ for $i \geq 1$ recursively by
\begin{equation}\label{eq:Yrecursion}
Y'_i =
\begin{cases}
    Y'_{i-1} & \text{if $X_i < \xi \vee Y'_{i-1} $}\\
    1 - X_{i} & \text{if $X_i \geq \xi \vee Y'_{i-1} $}.
\end{cases}
\end{equation}
Given the facts that have been accumulated, it turns out to be a reasonably easy task to prove a central limit theorem for
$A^o_n(\pi_\infty)$. One just needs to make the right connection to the known central limit theorems for
Markov additive processes.

To make this connection explicit, we first recall that, at any given time $1 \leq i \leq n$,
the decision maker knows the state of the selection process $Y'_{i-1}$ prior to time $i$,
and the decision maker also knows the value $X_i$ of the observation currently under consideration for selection.
The bi-variate random sequence $$\{ Z_i= (X_i, Y'_{i-1}): i = 1,2,3, \ldots \}$$ then represents the state of knowledge immediately
prior to the decision to accept or to reject $X_i$, and this sequence
may be viewed as a  Markov chain on the two-dimensional state space $\SS \equiv [0,1]\times [0, 1-\xi]$.
The Markov chain $\{ Z_i: i = 1,2,3, \ldots \}$ evolves over time according to a point-to-set
transition kernel that specifies the probability of moving from an arbitrary state $(x,y) \in \SS$
into a Borel set $C \subseteq \SS$ in one unit of time.
If we denote the transition kernel by $K((x,y), C)$, then we have the explicit formula
\begin{align*} 
K((x,y), C) & = \P\left( (X_{i+1}, Y'_{i}) \in C \, | \, X_i = x, Y'_{i-1} = y \right)\\
            & = \int_0^1 \!\!\! \big[ \1\{ (u,1-x) \in C \}\1(x \geq \xi \vee y) + \1\{ (u,y) \in C \} \1(x < \xi \vee y) \big]\, du,
\end{align*}
where the first summand of the integrand governs the transition when $X_i$ is chosen and the second summand
governs the transition when $X_i$ is rejected.
Given this explicit formula, it is straightforward (but admittedly a little tedious) to check that a
stationary probability measure for the kernel $K$ is given by the uniform distribution $\gamma$ on $\SS = [0,1]\times [0, 1-\xi]$.
We will confirm shortly that $\gamma$ is also the unique stationary distribution.

To more deeply understand the chain $Z_i$, $i=1,2, ...$, we now consider the
double chain $(Z_i, \bar{Z}_i)$, $i=1,2,\ldots$, where $Z_1=(x,y)$ is an arbitrary point of $\SS$
and $\bar{Z}_1$ has the uniform distribution on $\SS$. For $i=1,2,\ldots$, the
chains $\{Z_i=(X_i, Y'_{i-1})\}$ and  $\{\bar{Z}_i=(X_i, \bar{Y}'_{i-1}) \}$ share the same independent uniform sequence
$X_i$, $i=1,2,\ldots$, as their first coordinate, while their
second coordinates $Y'_{i-1}$ and  $\bar{Y}'_{i-1}$ are both determined by the recursion
\eqref{eq:Yrecursion}. Typically these coordinates differ because of their differing initial values, but we will check
that they do not differ for long.

To make this precise, we set $\nu=\min \{i \geq 1: X_i \geq 1 - \xi \}$, and we show that
$\nu$ is a \emph{coupling time} for
$(Z_i, \bar{Z}_i)$ in the sense that
$$
Z_i =\bar{Z}_i \quad \text{for all } i> \nu.
$$
Since $Y'_{i}$ and $\bar{Y}'_{i}$ both satisfy the recursion \eqref{eq:Yrecursion}, we have
$$
Y'_{i} \leq 1-\xi \quad \text{and} \quad \bar{Y}'_{i} \leq 1-\xi \quad \text{for all } i=1,2,\ldots,
$$
so by the definition of $\nu$, we must have
$$
\max \{ \xi \vee Y'_{\nu-1}, \xi \vee \bar{Y}'_{\nu-1} \} \leq X_\nu.
$$
The recursion \eqref{eq:Yrecursion} then gives us
$$
Y'_{\nu}=\bar{Y}'_{\nu}=1-X_\nu \quad \text{and} \quad Z_\nu = \bar{Z}_\nu.
$$
By the construction of the double process  $(Z_i, \bar{Z}_i)$, if one has
$Z_i(\omega) = \bar{Z}_i(\omega)$ for some $i=i(\omega)$, then $Z_j(\omega) = \bar{Z}_j(\omega)$ for all $j \geq i(\omega),$
so $\nu$ is indeed a coupling time for $(Z_i, \bar{Z}_i)$.

The coupling inequality \citeaffixed[p.~12]{Lin:DOVER2002}{see, e.g.,}
then tells us that for all Borel sets $C\subseteq \SS$, we have the total variation bound
\begin{equation}\label{eq:coupling-inequality}
\parallel K^\ell((x,y), C) - \gamma(C) \parallel_{\rm TV} \leq  \P(\nu > \ell) =  (1 - \xi)^{\ell},
\end{equation}
where $\gamma$ is the uniform stationary distribution on $\SS$.
The bound \eqref{eq:coupling-inequality} has several useful implications.
First, it implies that $\gamma$ is the \emph{unique} stationary distribution for the
chain with kernel $K$. It also implies \citeaffixed[Theorem 16.0.1]{MeyTwe:CUP2009}{see, e.g.,}
that the chain $\{ Z_i: i = 1,2,\ldots \}$ is uniformly ergodic; more specifically, it is a $\phi$--mixing chain with
$$
\phi(\ell) \leq 2 \rho^\ell \quad \text{and}  \quad \rho= 1-\xi.
$$

If we set $z = (x,y)$ and $f(z)=\1(x \geq y \vee \xi)$, then the representation \eqref{eq:A-infinity} can also be written
in terms the chain $\{Z_i: i=1,2,\ldots\}$ as
$$
A^o_{n}(\pi_\infty)=\sum_{i=1}^n f(Z_i),
$$
and this makes it explicit that $A^o_{n}(\pi_\infty)$ is a Markov additive process.
Our coupling and the uniform ergodicity of  $\{ Z_i: i = 1,2,\ldots \}$ imply
\citeaffixed[Theorem 17.5.3 and Lemma 17.5.1]{MeyTwe:CUP2009}{see, e.g.,}
that there is a constant $\sigma^2\geq 0$ such that
\begin{equation}\label{eq:variance-asymptotics}
\lim_{n \rightarrow \infty } n^{-1} \Var\left( A^o_{n}(\pi_\infty) \right) =
\lim_{n \rightarrow \infty } n^{-1} \Var_\gamma \left( A^o_{n}(\pi_\infty) \right)= \sigma^2,
\end{equation}
where the first variance refers to the chain started at $Z_1=(X_1,0)$ and the second variance
refers to the chain started at $Z_1$ with the stationary distribution $\gamma$ (i.e.~the uniform distribution on $\SS$).
The general theory also provides the series representation for the limit \eqref{eq:variance-asymptotics}:
\begin{align}\label{eq:sigma^2-Ainfty}
\sigma^2 & = \Var_\gamma\left[\1\left( X_1 \geq \big\{ \xi \vee Y'_{0} \big\} \right)\right] \\
& + 2 \sum_{i=2}^\infty \Cov_\gamma \left[\1\left( X_1 \geq \big\{ \xi \vee Y'_{0}\big\} \right),
\1\left( X_i \geq \big\{ \xi \vee Y'_{i-1}\big\} \right) \right], \nonumber
\end{align}
where the subscript $\gamma$ again refers to the situation in which the chain starts with $Z_1$ having the stationary distribution.

The general representations \eqref{eq:variance-asymptotics} and \eqref{eq:sigma^2-Ainfty} give us the existence of $\sigma^2$
but they do not automatically entail $\sigma^2 > 0$,
so to prove a central limit theorem for $A^o_{n}(\pi_\infty)$ with the classical normalization, one must independently
establish that $\sigma^2 > 0$. To show this, we first need an elementary lemma that provides a variance analog
to the information processing inequality for entropy.

\begin{lemma}[Information Processing Lemma]\label{lm:IP-var}
If a random variable $X$ has values in $\{1,2, \ldots \}$ and $P(X=1)=p$, then $p(1-p) \leq \Var (X)$.
\end{lemma}
\begin{proof} Define a function $f$ on the natural numbers $\N$ by setting $f(1)=0$ and $f(k)=1$ for $k > 1$. We then have
$|f(x)-f(y)|\leq |x-y|$ for all $x,y \in \N$. If we take $Y$ to be an independent copy of $X$, then we have
$$
2 p(1-p)=E[(f(X)-f(Y))^2] \leq E[(X-Y)^2]=2 \Var (X).
$$ \end{proof}

Now we can address the main lemma of this section.

\begin{lemma}\label{lm:variance-lower-bound}
There are constants $\alpha > 0$ and $N_* < \infty$ such that
$$
\alpha \, n  \leq \Var \left( A^o_{n}(\pi_\infty) \right) \quad \text{for all } n \geq N_*.
$$
\end{lemma}
\begin{proof}
We first set $\nu_0 \equiv 0$ and then define the stopping times
$$
\nu_t = \inf\{ i > \nu_{t-1}: \,  X_i \geq 1-\xi\}, \quad \quad t= 1,2, \ldots.
$$
We also set $ T(n) = \inf\{t :\,  \nu_t \geq n\}, $ and note that $T(n)$ is a stopping time with respect to the increasing sequence of $\sigma$-fields
\begin{equation*}\label{eq:def-G}
\G_t = \sigma\{\nu_1, \nu_2, \ldots, \nu_t\} \quad \text{for all } t \geq 1.
\end{equation*}
Next, we set
\begin{equation}\label{eq:def-U}
U_t = \sum_{i = \nu_{t-1}+1}^{\nu_t} \1\left( X_i \geq  \xi \vee Y'_{i-1}  \right) \quad \text{for } 1\leq t \leq T(n) \, \,\, \text{ and set}
\end{equation}
$$
V = \sum_{i=n+1}^{\nu_{T(n)}} \1\left( X_i \geq  \xi \vee Y'_{i-1}  \right),
$$
so we have the representation
\begin{equation}\label{eq:Ainfty-block-representation}
A^o_{n}(\pi_\infty)  = A^o_{\nu_{T(n)}}(\pi_\infty)  - V = \sum_{t=1}^{T(n)} U_t - V.
\end{equation}
Here, the random variables $U_t$, $t=1,2,\ldots$, are independent and identically distributed. We also have $V \leq \nu_{T(n)} - n$ and
$\nu_{T(n)} = \inf\{ i \geq n: X_i \geq 1-\xi\}$, so the
variance of $V$ is bounded by a constant that depends only on $\xi$.
The existence of the limit \eqref{eq:variance-asymptotics} and the Cauchy-Schwarz inequality then give us
\begin{equation}\label{eq:VarAn-VarAnu}
\Var \left( A^o_{n}(\pi_\infty) \right) = \Var\left( A^o_{\nu_{T(n)}}(\pi_\infty) \right) + O(\sqrt{n} ) \quad \quad \text{ as } n \rightarrow \infty,
\end{equation}
so to prove the lemma it suffices to obtain a linear lower bound for $\Var( A^o_{\nu_{T(n)}}(\pi_\infty) )$.

By the definition of $\nu_{T(n)}$ and $U_t$, $t=1,2,\ldots$, we have
$$
A^o_{\nu_{T(n)}}(\pi_\infty) = \sum_{t=1}^{T(n)} U_t
$$
so, by the conditional variance formula, the independence of the $U_t$'s,
and the fact that ${T(n)}$ is $\G_{T(n)}$ measurable,
we have the bound
\begin{equation}\label{eq:variance-lower-bound}
\Var\big(  \sum_{t=1}^{T(n)} U_t \big) \geq \E\big[ \Var\big(  \sum_{t=1}^{T(n)} U_t \, | \, \G_{T(n)} \big) \big]
= \E\big[ \sum_{t=1}^{T(n)} \Var\big( U_t \, | \, \G_{T(n)} \big) \big].
\end{equation}
We now note from the definition \eqref{eq:def-U} that $U_t$ takes values in  $\{1, 2, \ldots, \nu_t- \nu_{t-1}\}$.
Thus, if $p$ is the probability  that no $X_i$ is selected for $i \in \{ \nu_{t-1}+1, \ldots, \nu_t-1 \}$, then setting
$a = (1 - \xi)^{-1} \xi $, we have
$$
p = \P( U_t=1 \,|\, \G_{T(n)})=\P\left(X_i< \xi \,\, \forall \,\, \nu_{t-1}+1 \leq i \leq \nu_t-1 \, | \, \G_{T(n)} \right) = a^{\nu_t - \nu_{t-1} - 1 }.
$$
Now, by applying Lemma \ref{lm:IP-var} to the conditional expectation, we have
$$
\Var\big( U_t \, | \, \G_{T(n)} \big) \geq a^{\nu_t - \nu_{t-1} - 1 }\left( 1 - a^{\nu_t - \nu_{t-1} - 1 } \right),
$$
so from \eqref{eq:variance-lower-bound}, we have
$$
\Var\big(  \sum_{t=1}^{T(n)} U_t \big) \geq  \E\big[ \sum_{t=1}^{T(n)} a^{\nu_t - \nu_{t-1} - 1 }\left( 1 - a^{\nu_t - \nu_{t-1} - 1 } \right) \big].
$$
The summands are independent and identically distributed and ${T(n)}$ is a stopping time
with respect to the increasing sequence of $\sigma$-fields $\G_t = \sigma\{\nu_1, \nu_2, \ldots, \nu_t\}$, $t \geq 1$,
so by Wald's identity, we have
\begin{equation}\label{eq:variance-lower-bound-Anu}
\Var\big(  \sum_{t=1}^{T(n)} U_t \big) \geq \E\left[ T(n) \right] \E\left[ a^{\nu_1 - 1 }\left( 1 - a^{\nu_1 - 1 } \right)\right].
\end{equation}
For the stopping time ${T(n)}$, we have the alternative representation
$$
{T(n)} = 1 + \sum_{i=1}^{n-1} \1\left( X_i \geq 1-\xi \right),
$$
so we have $ \E\left[{T(n)}\right] = \xi \, n + O(1)$. Since $\nu_1$ has the geometric distribution with success probability $\xi$,
we also have $\E\left[ a^{\nu_1 - 1 }\left( 1 - a^{\nu_1 - 1 } \right)\right] > 0$, so by \eqref{eq:VarAn-VarAnu} and
\eqref{eq:variance-lower-bound-Anu} the proof of the lemma is complete.
\end{proof}

All of the pieces are now in place. By the central limit theorem for functions of uniformly ergodic Markov chains
(\citename{MeyTwe:CUP2009}, \citeyear*{MeyTwe:CUP2009}, Theorem 17.5.3; or \citename{Jon:PS2004}, \citeyear*{Jon:PS2004}, Corollary 5),
we get our central limit theorem for $A^o_{n}(\pi_\infty)$.

\begin{proposition}[Central Limit Theorem for $A^o_{n}(\pi_\infty)$]\label{pr:clt-Ainfty}
As $n \rightarrow \infty$, we have the limit
$$
\frac{A^o_{n}(\pi_\infty) - \mu \, n }{\sqrt{n}} \Longrightarrow N(0, \sigma^2),
$$
where $\mu = \E_\gamma \left[ \1\left( X_1 \geq \{ \xi \vee Y'_{0}\} \right) \right]$,
$\gamma$ is the stationary distribution for the Markov chain $\{Z_i: i=1,2,\ldots\}$,
and $\sigma^2$ is the constant defined by either the limits \eqref{eq:variance-asymptotics} or the sum \eqref{eq:sigma^2-Ainfty}.
\end{proposition}

By appealing to the known relation \eqref{eq:MeanAsymptotics} that  $\E[A^o_n(\pi^*_n)]  =  (2 - \sqrt{2}) n + O(1)$,
one can show with a bit of calculation that here we have $\mu=2-\sqrt{2}$. Since this identification is
implicit in the calculations of the next section,
there is no reason to belabor it here.

\section{$A_n^o(\pi^*_n)$ and $A_n^o(\pi_\infty)$ are Close in $L^2$} \label{se:CLT-optimal-A}

Proposition \ref{pr:clt-Ainfty} tells us that the easy sum $A^o_{n}(\pi_\infty)$
obeys a central limit theorem, and now the task is to show that the harder sum $A^o_{n}(\pi^*_n)$ follows the same law.
The essence is to show that, after centering, the random variables
$A^o_{n}(\pi^*_n)$ and $A^o_{n}(\pi_\infty)$ are close in $L^2$
in the sense that $\parallel A^o_n(\pi^*_n) - A^o_n(\pi_\infty) - \E\left[ A^o_n(\pi^*_n) - A^o_n(\pi_\infty)  \right] \parallel_2 = o(\sqrt{n})$
as $n \rightarrow \infty$.
For technical convenience, we work with the random variable
$$
\Delta_n \stackrel{\rm def}{=} A^o_{n-2}(\pi^*_n) - \E\left[ A^o_{n-2}(\pi^*_n) \right]
- A^o_{n-2}(\pi_\infty) + \E\left[ A^o_{n-2}(\pi_\infty) \right].
$$
The essential estimate of our development is given by the next lemma. In one way or another,
the proof of the lemma calls on all of the machinery that has been developed.

\begin{lemma}[$L^2$-Estimate] There is a constant $C$ such that, for all $n \geq 3$, we have
$$
\parallel \Delta_n \parallel_2^2 \leq C \sum_{k=3}^{n} (\xi - \xi_{k});
$$
so, in particular, we have the asymptotic estimate
$$
\parallel \Delta_n \parallel_2 = o(\sqrt{n}) \quad \quad \text{ as } n \rightarrow \infty.
$$
\end{lemma}
\begin{proof}
We first note that the threshold function lower bound \eqref{eq:lower-bound-gk-inLemma} implies that $Y_i \leq 5/6$ for all $1 \leq i \leq n-2$.
Consequently, if $X_i \geq 5/6$, then $X_i$ is selected by both of the policies $\pi^*_n$ and $\pi_\infty$.
At such a time $i$, we have a kind of ``renewal event," though we still have to be attentive to the non-homogeneity of
the selection process driven by $\pi^*_n$.

To formalize this notion, we set $\tau_0 = 0$ and, for $m\geq 1$, we define stopping times
$$
\tau_m = \inf\left\{ i> \tau_{m-1} : \, X_i \geq 5/6 \right\} \quad\quad \text{and} \quad \quad \tau'_m = \min\{ \tau_m, n-2 \};
$$
so $\tau_m$ is the time at which the $m$th ``renewal'' is observed.
For each $1 \leq j \leq n-2$, we then set
$$
N(j) = \sum_{i=1}^j \1( X_i \geq 5/6 ),
$$
so the time $\tau_{N(j)}$ is the time of the last renewal up to or equal to $j$, the time
$\tau_{N(j)+1}$  is the time of the first renewal strictly after $j$, and we have the inclusion
$$
\tau_{N(j)}\leq j < \tau_{N(j)+1}.
$$
For $1\leq j \leq n-2$, we then consider the martingale differences defined by
$$
d_j= \E\left[ A^o_{n-2}(\pi^*_n) - A^o_{n-2}(\pi_\infty) | \F_j \right]- \E\left[ A^o_{n-2}(\pi^*_n) - A^o_{n-2}(\pi_\infty) | \F_{j-1} \right],
$$
where $\F_0$ is the trivial $\sigma$-field and $\F_j = \sigma\{X_1,X_2,  \ldots, X_j\}$ for $1\leq j \leq n$.
Using the counting variables
$$\eta_i \equiv \1\left( X_i \geq g_{n-i+1}(Y_{i-1}) \right)
\quad \text{and} \quad
\eta'_i \equiv \1\left( X_i \geq \xi \vee Y'_{i-1}  \right),
$$
we have the tautology
\begin{align}\label{eq:dj-representation-break-sum}
d_j  & = \E [ \sum_{i=j}^{\tau'_{N(j)+1}} (\eta_i- \eta'_i) \, | \, \F_{j} ]
       - \E [ \sum_{i=j}^{\tau'_{N(j)+1}} (\eta_i- \eta'_i) \, | \, \F_{j-1} ] \\
     & + \E [ \sum_{i=\tau'_{N(j)+1} + 1}^{n-2} \!\!\!\!(\eta_i- \eta'_i) \, | \, \F_{j} ] \nonumber
       - \E [ \sum_{i=\tau'_{N(j)+1}+1}^{n-2} \!\!\!\! (\eta_i- \eta'_i) \, | \, \F_{j-1} ],
\end{align}
and this becomes more interesting after one checks that the last two terms cancel.

To confirm the cancelation, we first recall that, for $\tau_{N(j)+1} < n-2$, the value $X_{\tau_{N(j)+1}} \geq 5/6$
is selected as a member of the alternating subsequence under both policies $\pi^*_n$ and $\pi_\infty$, so we also have
$$
Y_{\tau_{N(j)+1}} = Y'_{\tau_{N(j)+1}} = 1 - X_{\tau_{N(j)+1}}.
$$
Any difference in the selections that are made by the policies $\pi^*_n$ and $\pi_\infty$ after time $\tau_{N(j)+1}$
is measurable with respect to the $\sigma$-field
$$
\T_j = \sigma\{X_{\tau_{N(j)+1}}, X_{\tau_{N(j)+1}+1}, \ldots, X_{n-2}\}.
$$
Trivially, we have $j < \tau_{N(j)+1}$, so $\F_j$ is independent of  $\T_j$, and
the last two addends in \eqref{eq:dj-representation-break-sum} do cancel as claimed.

We can therefore write
\begin{equation}\label{eq:dj-stopped}
d_j = \E [ \, \sum_{i=j}^{\tau'_{N(j)+1}} (\eta_i- \eta'_i) \, | \, \F_{j} ]
       - \E [ \, \sum_{i=j}^{\tau'_{N(j)+1}} (\eta_i- \eta'_i) \, | \, \F_{j-1} ] =  W_j - I_{j-1}(W_j),
\end{equation}
where $W_j$ denotes the first summand and $I_{j-1}$ is the projection onto $L^2(\F_{j-1})$.
Denoting the identity by $I$, we have that  $I-I_{j-1}$ is an $L^2$ contraction, so
\begin{equation}\label{eq:difference-MDS-second-moment}
\E \left[ d^2_j \right] \leq \E \left[ W_j^2 \right]=\E \big[ \big( \sum_{i=j}^{\tau'_{N(j)+1} } ( \eta_i - \eta'_i ) \big)^2 \big],
\end{equation}
and the remaining task is to estimate the last right-hand side.

For $1\leq j \leq n-2$, we let $L(j)$ denote time from $j$ since the last renewal preceding $j$; in other words, $L(j)$ is
the \emph{age} at time $j$. Analogously, we let $M(j)$ denote
the time from $j$ until the time of the next renewal or until time $n-2$;  so $M(j)$ is the \emph{residual life} at time $j$ with truncation at time
$n-2$. We then have
$$
L(j) = j - \tau_{N(j)} \quad \quad \text{ and } \quad \quad  M(j) = \tau'_{N(j) + 1} - j.
$$
Our interarrival times are geometric, so $L(j)$ and $M(j)$ are independent, and for $p = 1/6$ we have
$$
\P(L(j) = \ell ) =
\begin{cases}
    p (1-p)^{\ell}   & \quad \text{if $0 \leq \ell < j$} \\
    (1-p)^{j}       & \quad \text{if $\ell = j$,}
\end{cases}
$$
and
$$
\P(M(j) = m ) =
\begin{cases}
    p (1-p)^{m -1}  &  \quad \text{ if $1 \leq m < n-2-j$} \\
    (1-p)^{n-3-j}   &  \quad \text{ if $m = n-2-j$.}
\end{cases}
$$

We now introduce the \emph{disagreement} set
$$
D_j[\ell, m] = \left\{\omega: \exists \, i \in \{j-\ell + 1, \ldots, j, \ldots, j+m \} : \, X_i(\omega) \in [\xi_{n-i+1}, \xi] \right\};
$$
this is precisely the set of $\omega$ for which,
if $Y_{j-\ell} = Y'_{j-\ell}$, then the policies $\pi_\infty$ and $\pi_n^*$ differ in at least one
selection during the time interval $\{ j-\ell+1, \ldots, j+m\}$,
while on the complementary set $ D_j^c [\ell, m]$ the selections all agree.
Thus, by the crudest possible bound, we have
$$
| \sum_{i=j}^{\tau'_{N(j)+1 }} ( \eta_i - \eta'_i ) | \leq ( L(j) + M(j)  ) \1 \left( D_j[ L(j),M(j)] \right),
$$
and when we square both sides and rearrange, we obtain
\begin{align}
\bigg( \sum_{i=j}^{\tau'_{N(j)+1 }} ( \eta_i - \eta'_i ) \bigg)^2
    & \leq ( L(j) + M(j)  ) ^2 \1\left( D_j[ L(j),M(j)] \right) \nonumber \\
    & = \sum_{m=1}^{n-2-j} \sum_{\ell = 0}^j (\ell + m)^2 \1\left(D_j[ \ell,m] \right) \1(L(j) = \ell) \1(M(j) = m). \label{eq:squared-difference2}
\end{align}
For each $1 \leq j \leq n-2$, we now set
$$
R_j[\ell, m] = \left\{ \omega: X_i(\omega) < 5/6 \text{ for all } i \in \{ j-\ell +1, \ldots, j+m \} \right\},
$$
so, $R_j[\ell, m]$ is the event that no renewal takes place in $[j - \ell + 1, j]$ or in $[j+1, j+m]$.
By the definition of $L(j)$ and $M(j)$, we then have
$$
\1(L(j) = \ell) = \1 \left( R_j[\ell, 0]\right) \1\left(X_{j-\ell}\geq 5/6 \text{ or } \ell = j\right),
\quad \quad \text{ for } 0 \leq \ell \leq j,
$$
and
$$
\1(M(j) = m) \leq \1 \left( R_j[0, m-1]\right), \quad \quad \text{ for } 1 \leq m \leq n-2-j.
$$
Thus, if we define $\1 \left( R_j[ 0, 0]\right) \equiv 1$, then  we have the composite bound
\begin{equation}\label{eq:IndicatorBound}
\1(L(j) = \ell) \1(M(j) = m)  \leq \1 \left( R_j[\ell, m-1]\right)\1\left(X_{j-\ell}\geq 5/6 \text{ or } \ell = j\right),
\end{equation}
so by inserting \eqref{eq:IndicatorBound} in \eqref{eq:squared-difference2} and recalling \eqref{eq:difference-MDS-second-moment}, we find
\begin{equation}\label{eq:bound-MDS-squared}
\E\left[ d_j^2 \right] \! \leq \!\!\!
\sum_{m=1}^{n-2-j}  \sum_{\ell = 0}^j (\ell + m)^2
\E \left[ \1\!\left( D_j[\ell,m] \right) \1\!\left( R_j[ \ell, m \!-\! 1] \right) \1\!\left(X_{j-\ell}\geq 5/6 \text{ or } \ell \! = \! j\right)\right].
\end{equation}
The expected value on the right-hand side of \eqref{eq:bound-MDS-squared}
accounts for the probability that policies $\pi^*_n$ and $\pi_\infty$ differ when one renewal
has occurred at time $j-\ell$, and no renewal will occur until time $j+m$.
For this to happen, we need at least one $i \in \{j-\ell+1, \ldots, j+m\}$
such that $X_i \in [\xi_{n-i+1}, \xi]$.
Since the $X_i$'s are uniformly distributed on $[0,1]$, the probability that $X_i \in [\xi_{n-i+1}, \xi]$
equals $\xi - \xi_{n-i+1}$ and,
by the monotonicity of the minimal fixed points in Lemma \ref{lm:xi_k},
we have the upper bound $\xi - \xi_{n-i+1} \leq \xi - \xi_{n-(j+m)+1}$ for all $i \in \{ j-\ell+1, \ldots, j+m\}$.
Then, we can estimate the right-hand side of \eqref{eq:bound-MDS-squared} with Boole's inequality,
and obtain that there is a constant $C$ such that
\begin{align*}
&\E \left[ \1\!\left( D_j[\ell,m] \right) \1\!\left( R_j[ \ell, m \!-\! 1] \right)
\1\!\left(X_{j-\ell}\geq 5/6 \text{ or } \ell \! = \! j\right)\right] \\
&\quad \quad \leq  C (m - \ell) \! \left(\xi - \xi_{n - (j + m) + 1} \right)\! (1-p)^{\ell + m - 1}.
\end{align*}
At this point, $C=6/5$ would suffice, but subsequently $C$ denotes a Hardy-style constant that may change from line to line.
If we use this last bound in \eqref{eq:bound-MDS-squared}, we obtain
$$
\E\left[ d_j^2 \right] \leq  C \sum_{m=1}^{n-2 - j} \sum_{\ell = 0}^j (\ell + m )^3 (\xi - \xi_{n-(j+m)+1}) (1-p)^{\ell + m - 1},
$$
so, if we change variable by applying the transformation $r = j+m$, we have
$$
\E\left[ d_j^2 \right] \leq C \sum_{r=j+1}^{n-2 } \sum_{\ell = 0}^j (\ell + r - j )^3 (\xi - \xi_{n-r+1}) (1-p)^{\ell + r - j  - 1}.
$$
If we now sum over $1 \leq j \leq n-2$, we obtain
$$
\E\left[ \Delta_n ^2 \right] = \sum_{j=1}^{n-2} \E\left[ d_j^2 \right]
\leq C \sum_{j=1}^{n-2} \sum_{r=j+1}^{n-2 } \sum_{\ell = 0}^j (\ell + r - j )^3 (\xi - \xi_{n-r+1}) (1-p)^{\ell + r - j  - 1},
$$
so if we interchange the first with the second sum and rearrange, we have
$$
\E\left[ \Delta_n ^2 \right] \leq C\sum_{r=2}^{n-2}  (\xi - \xi_{n-r+1}) \big\{  \sum_{j=1}^{r-1}  \sum_{\ell = 0}^j (\ell + r - j)^3 (1-p)^{\ell + r - j - 1} \big\}.
$$
At this point, it is elementary to check that for all $r$ the last double sum is bounded by
the constant $\sum_{u=1}^\infty u^4 (1-p)^{u-1}$,
and this completes the proof of our lemma.
\end{proof}

\section{Some Perspective} \label{se:conclusions}

We have pursued the proof of a specific central limit theorem, but some aspects of our analysis may have useful implications
for a wider class of Markov decision problems (MDPs).
For example, we took advantage here of the existence of a
policy $\pi_\infty$ that could be viewed heuristically as the ``optimal policy at infinity,'' and the temporal homogeneity of
this policy then gave us access to the machinery of Markov additive processes. Many MDPs offer similar prospects.

To be sure, specialized efforts were needed to relate the finite horizon policy $\pi_n^*$ to the limiting policy, but
the pattern used here does offer some general guidance. In almost any MDP, the Bellman equation gives one good prospects for computing
the value function, but to extract the full value of those functions one needs to
develop a deeper understanding of their geometry --- and the geometry of the associated threshold functions.
Here, the development of such an understanding would have been stymied without the guidance provided by Figure \ref{fig:thresholds}.
If one views our analysis as a case
study, then one message is that when facing a new
MDP one would almost always be wise to begin with the best numerical work that the problem allows.

Finally, the Bellman equation grants a natural place for induction in the analysis of many MDPs, and here we have
seen that such inductions can be greatly helped by various forms of diminishing returns.
Without the special properties represented by \eqref{eq:diminishing-return} and \eqref{eq:DiminishingReturns2}
our inductions could not have moved forward.
One can anticipate that some aspect of this experience will be present in the analysis of many other MDPs.

\section*{Acknowledgement}

The authors are pleased to thank Alexandre Belloni, the editor, and the referees for useful comments on an earlier draft of this paper.

\bibliographystyle{agsm}

\end{document}